\documentclass{article}

\usepackage[final,nonatbib]{nips}
\usepackage[utf8]{inputenc} 
\usepackage[T1]{fontenc}    
\usepackage{hyperref}       
\usepackage{url}            
\usepackage{booktabs}       
\usepackage{amsfonts}       
\usepackage{nicefrac}       
\usepackage{microtype}      

\usepackage{subcaption}
\usepackage{amsmath}
\usepackage{amsthm}
\usepackage{algorithm}%
\usepackage{algpseudocode}
\usepackage{braket}
\usepackage{graphicx}

\newtheorem{definition}{Definition}
\newtheorem{remark}{Remark}
\newtheorem{theorem}{Theorem}
\newtheorem{lemma}{Lemma}
\newtheorem{assumption}{Assumption}
\newcommand{\norm}[1]{\left\lVert#1\right\rVert}

\newcommand{\bI}{\mathbf{I}}
\newcommand{\bX}{\mathbf{X}}
\newcommand{\bR}{\mathbf{R}}

\title{Stochastic Second-Order Optimization via von Neumann Series}

\author{
	Mojm\'ir Mutn\'y \\
	Department of Mathematics\\
	ETH Z\"urich\\
	Z\"urich, Switzerland\\
	\texttt{mmutny@student.ethz.ch} \\
}

\begin{document}
	
	\maketitle	
	\begin{abstract}
		A stochastic iterative algorithm approximating second-order information using von Neumann series is discussed. We present convergence guarantees for strongly-convex and smooth functions. Our analysis is much simpler in contrast to a similar algorithm and its analysis, LISSA. The algorithm is primarily suitable for training large scale linear models, where the number of data points is very large. Two novel analyses, one showing space independent linear convergence, and one showing conditional quadratic convergence are discussed. In numerical experiments, the behavior of the error is similar to the second-order algorithm L-BFGS, and improves the performance of LISSA for quadratic objective function. 
	\end{abstract}
	
	\section{Introduction}
	
	With the advent of ''Big Data'' age, the need for novel, fast and robust optimization algorithms becomes ever more prevalent. The two most common classes of optimization algorithms are first-order methods which utilize only gradient information, and on the other hand second-order methods, which incorporate curvature information from the objective function. In this paper, we explore the possibility of using the von Neumann approximation series as a tool to iteratively build up second-order information for efficient optimization methods.
	
	Many supervised learning problems can be formulated as empirical risk minimization, 
	\begin{equation}
		F(x) :=  \frac{1}{n}\sum_{i=1}^{n} f(x,z^i,y^i).
	\end{equation}
	where $\{z^i,y^i\}$ are data points lying in $\mathbb{R}^d$, with associated labels $y^i \in \mathbb{R}$. The optimization variable $x \in \mathbb{R}^d$ is the parametrization of the model that one seeks to determine by solving the following unconstrained minimization problem, 
	\begin{equation}
		x^* :=   \arg \min_{x \in \mathbb{R}^d} F(x).
	\end{equation}
	
	We focus on the scenario when the number of data points is much larger than the dimension of the problem i.e. $n \gg d$. In this scenario, the methods that belong to standard optimization routines are SGD \cite{Bottou2010}, dual SDCA \cite{Shalev-Shwartz2013} or more advanced algorithms such as SVRG \cite{Johnson2013}, SAGA \cite{Defazio2014}, SAG \cite{Schmidt2013}, or S2GD \cite{Konecny2014}, which collectively belong to the school of first-order methods. In Big Data optimization, the school of second-order optimization is represented by L-BFGS \cite{Liu1989} algorithm, and many of its varieties. Apart from L-BFGS, recently, ideas of sub-sampling the Hessian formed a plethora of algorithms in literature \cite{Roosta-Khorasani2016,Roosta-Khorasani2016a,Qu2015,Pilanci2015}.
	
	We propose a method where access to unbiased estimates of a Hessian $\{ \mathbf{X}_i\}_{i=1}^n$ is assumed. These estimates can be used to iteratively build the second-order information by exploiting the well known von Neumann series approximation \eqref{eq:von_neumann}. In fact, our method builds asymptotically inverse Hessian in the expectation. One can contrast this approach to SGD, where gradients are sampled, and in expectation true gradient is built. In our analysis, we assume exact gradient information and the Hessian matrix is sampled. However, an extension to inexact first-order information could be feasible as well.
	
	\begin{remark}[\cite{Chen2005}]For a matrix $\mathbf{A} \in \mathbb{R}^{d\times d}$ such that $\mathbf{A} \succeq 0 $ and $\norm{\mathbf{A}}_2 \leq 1$ we have that
		\begin{equation}\label{eq:von_neumann}
			\mathbf{A}^{-1} = \sum_{j=0}^{\infty} (\mathbf{I}-\mathbf{A})^j.
		\end{equation}
	\end{remark}
	
	\subsection{LISSA algorithm}
	A very similar idea to ours has been presented in \cite{Agarwal2016} with slight differences in the algorithm formulation. The main focus of \cite{Agarwal2016} was on averaging the estimators of Hessian to gain better estimation properties utilizing the theory of concentration inequalities. Their analysis of the algorithm requires warm start by gradient descent to prove linear convergence to the optimum. The algorithm LISSA from \cite{Agarwal2016} will not be reviewed here, but essentially it consist of two parts parametrized by two parameters, $S_1$ and $S_2$. The parameter $S_2$ denotes the number of Hessian samples used to build one estimate and $S_1$ denotes number of repetitions of this step. In numerical experiments, however, authors set $S_1 = 1$ corresponding to the absence of averaging. 
	
	\subsection{Contributions}
	In this work, we present two new simplified convergence analyses of the stochastic second-order algorithm that builds approximation of a inverse Hessian using unbiased samples of a Hessian matrix. We show more intuitive convergence rates than a similar algorithm \cite{Agarwal2016} under milder mathematical conditions. Namely, we do not require initial convergence to a closed proximity of an optimum to show linear convergence, and in addition, we show conditional quadratic convergence given a specific condition on a gradient of the objective. Our formulation of the algorithm does not require the notion of averaging, and thus reduces the number of parameters needed to perform the theoretical analysis. Additionally, the theoretical formulation of our algorithm is more closely aligned with its practical execution.
	
	\section{Algorithm}
	In order to define the algorithm formally, we first have to introduce a concept of an index sampling i.e. set-valued random variable. We closely follow the notation and theory of sampling of \cite{Qu2014}. Furthermore, we propose our first two crucial assumptions. The Assumption \ref{ass:scaled} might seem unnatural but we note that most of the problems can be recast such that it holds.
	
	\begin{definition}[ordered $\tau$-independent sampling]
		An ordered $\tau$-independent sampling $\hat{S}$ is an ordered random set-valued mapping with elements from $[n]$ and with cardinality $\tau$. In addition, the possible outcomes of this random variable are all equally probable.
		
	\end{definition}
	
	\begin{assumption}[Hessian]\label{ass:scaled} The function $F(x) : \mathbb{R}^d \rightarrow \mathbb{R}$ is twice continuously differentiable, and has scaled Hessian s.t. $\nabla^2 F(x) \preceq \mathbf{I}$ $\forall x \in \mathbb{R}^d$. 
	\end{assumption}

	\begin{assumption}[Hessian Samples]\label{ass:samp} We assume existence of $\{\mathbf{X}_i\}_{i=1}^n$, with each $\mathbf{X}_i \in \mathbb{R}^{d\times d}$, at each point $x$ where each $\mathbb{E}[\mathbf{X}_i] = \nabla^2 F(x) $ for a given $x$. To signify the $x$ dependence, we sometimes use $\bX_j(x)$.
	\end{assumption}
	
	The Algorithm is presented in \ref{alg:Method1}. The step-size is fixed before. The usual parameter used is $c = 1$. Additionally, we present two versions of the algorithm, \emph{practical} and \emph{theoretical}. If Hessians samples are independent of $x^k$ the \emph{practical} version of the algorithm can be used, where previous estimates of Hessian can be reused. In the other, more general case, a full estimator has to be rebuilt from the start, which can be grossly inefficient. In such cases, the series must be truncated for practical purposes. 
	
	\begin{algorithm}
		\caption{Iterative Stochastic Second-Order Algorithm (ISSA) }
		\label{alg:Method1}
		\renewcommand{\algorithmicrequire}{\textbf{Parameters:}}
		\renewcommand{\algorithmicensure}{\textbf{Initialization:}}
		\begin{algorithmic}[1]
			\Require Sampling	 $\hat{S}$, oracle to return $\bX_j(x)$ $\forall j$ and $x$, $c \in \mathbb{R} $ according to the Theorems \ref{thm:theorem1} and \ref{thm:theorem2}, or adaptive. 
			\Ensure Pick $x_0 \in \mathbb{R}^d$, $\mathbf{R}^0 = \mathbf{I}$, $S^0 = \{\}$
			\For{ $k = 1,2, \dots $ }
			\State Sample $\hat{S}$ to get $S^k = S^{k-1} \cup \hat{S}$
			\If{Hessian is constant}
			\State $\mathbf{R}^k = \mathbf{R}^{k-1}$
			\For{ $ j \in \hat{S}$ }
			\State $ \mathbf{R}^k = \mathbf{I} + (\mathbf{I}-\bX_j )\mathbf{R}^{k}$ 	
			\EndFor
			\Else 
			\State $\mathbf{R}^{k} = \mathbf{I}$
			\For{ $ j \in S^k$ }
			\State $ \mathbf{R}^k = \mathbf{I} + (\mathbf{I}-\bX_j(x^k))\mathbf{R}^{k}$ 	
			\EndFor	
			\EndIf
			\State $x^{k+1} \leftarrow x^k - \frac{1}{c}\mathbf{R}^k\nabla f(x^k)$
			\EndFor
			
		\end{algorithmic}
	\end{algorithm}
	
	A concrete example of the empirical risk minimization problem which fulfills Assumptions \ref{ass:scaled} and \ref{ass:samp} is ridge regression fitting with a cost function $\frac{1}{2n} \sum_{i=1}^{n} (z_i^\top x - y_i)^2 + \frac{\lambda}{2}\norm{x}_2$. The Hessian can be easily calculated $\nabla^2 F(x) = \left(\frac{1}{n}\sum_{i=1}^{n} z_iz_i^\top + \lambda \mathbf{I}\right)$, and the Hessian samples enumerated with $i \in [n]$ can be defined as $\mathbf{X}_i := z_iz_i^\top + \mathbf{I}\lambda. $ This problem is a prime example, where practical version of the algorithm can be used as $\bX_i$ are independent of $x^k$.
	\section{Convergence Analysis}
	
	\subsection{Assumptions}
	\begin{assumption}[Smoothness, Strong Convexity]\label{ass:smooth}
		There exists a positive constant $\alpha, \beta \leq 1$ such that $\forall x,h \in \mathbb{R}^d$,
		\begin{equation}\label{eq:strgcnvx}
			F(x)+\braket{\nabla F(x),h}+\frac{\alpha}{2}\norm{h}^2_2   \leq f(x+h)\leq F(x)+\braket{\nabla F(x),h}+\frac{\beta}{2}\norm{h}^2_2.
		\end{equation}
		Minimizing the above on both sides in $h$ gives:
		\begin{equation}\label{eq:strgcnvx2}
			F(x)-f(x^*)\leq\frac{1}{2\alpha}\braket{\nabla F(x),\nabla F(x)}
		\end{equation}
		where $x^*$ denotes the minimum point of $f$.
	\end{assumption}
	%
	
	\subsection{Basic Properties}
	One of the contributions of our work is introduction of a new notation to prove the convergence results presented in Definition \ref{def:2}.
	\begin{definition}[Update]\label{def:2}
		Let $\mathbf{X}_i$ be a matrix valued random variable such that $\mathbb{E}[\mathbf{X}_i] = \nabla^2 f $ and $S$ be an ordered set such that $|S|=k\tau$ with element from $[n]$, then define
		\[
		\label{eq:Rk}
		\mathbf{R}^k :=  (\mathbf{I} \text{ } \mathbf{0}) \left( \prod_{j \in S}  \begin{pmatrix}
		(\mathbf{I}-\mathbf{X}_j) & \mathbf{I} \\
		\mathbf{0} & \mathbf{I} 
		\end{pmatrix} \right) (\mathbf{I} \text{ } \mathbf{I} ) ^\top.
		\]
	\end{definition}
	\begin{lemma}
		Let us have the same assumptions as in Definition \ref{def:2}, 
		\begin{equation} \label{eq:234}
			\mathbb{E}[\mathbf{R}^k] = \sum_{j=1}^{k\tau} (\bI - \nabla^2 F)^j \leq \sum_{j=1}^{\infty} (\bI - \nabla^2 F)^j = (\nabla^2 F)^{-1}
		\end{equation}
		and 
		\begin{equation} \label{eq:approx}
		\norm{ \mathbb{E}[\mathbf{R}^k] - (\nabla^2 F)^{-1} }_2^2 \leq \frac{(1-\alpha)^{k\tau}}{\alpha}.
		\end{equation}
	\end{lemma}
	
	\subsection{Linear Convergence}
		\begin{theorem} \label{thm:theorem1}
		Let Assumptions  \ref{ass:scaled}, \ref{ass:samp} and \ref{ass:smooth} be satisfied, then 
		Algorithm \ref{alg:Method1} converges linearly with the rate
		\begin{equation}\label{eq:mu}
		\mu  := \left( \frac{ (1-\alpha)^4 \alpha}{\beta ( 2-\alpha   )^2 ( (2-\alpha) + \alpha^2(1-\alpha)^{k\tau}} \right),
		\end{equation} 
		with $c := \frac{\beta (2-\alpha)^2 + \beta(2-\alpha)(1-\alpha)^{k\tau}}{ (1-\alpha)^2 - (1-\alpha)^{2k\tau}}$. \\Alternatively, 
		\begin{equation}
		\mathbb{E}[F(x^{k+1})-F(x^k)]  \leq  - \mu (F(x^k) - F(x^*)).
		\end{equation}
		\end{theorem}
		
	Note than in contrast to the results of \cite{Agarwal2016}, we prove linear convergence without the need to converge to sufficient vicinity of the optimum. However, the result is much worse than a standard gradient descend with rate $\frac{\alpha}{\beta}$ by at least the factor of $8$. 
	\subsection{Quadratic Convergence}
	\begin{theorem}\label{thm:theorem2}
		Let Assumptions  \ref{ass:scaled}, \ref{ass:samp} and \ref{ass:smooth} be satisfied, and additionally let
		\begin{equation}\label{eq:cond}
			\frac{\alpha^2}{8\beta} \geq \norm{\nabla F(x^k)} \geq  \frac{\beta (1-\alpha)^{k\tau}}{2} + \frac{\alpha \beta (\beta-\alpha)}{4},
		\end{equation}
		then the algorithm improves the number of significant digits quadratically when run with $c = 1$.
	\end{theorem}
	
	We would like to remark that the last condition for Theorem \ref{thm:theorem2} is rather artificial, and is hardly satisfied for the whole duration of the optimization; if at all. However, empirical results from \cite{Agarwal2016}, as well as ours, show a 	 very steep initial convergence to the optimum, which might hint that in such circumstances the conditions of the theorem is satisfied and is the consequence of the initial rapid decrease. For example, consider, $\tau$ large , $\beta = 1$ and $\alpha \geq 1/3$, then we see that a suitably chosen starting point could lead to gradient fulfilling the condition allowing for very fast convergence.
	
	\section{Further Considerations}
		\subsection{Batch gradient evaluation}
			The ISSA algorithm requires the evaluation of full gradient at each step as presented in Algorithm \ref{alg:Method1}. This is significantly hinders the applicability of this algorithm as the evaluation of the gradient can be as expensive as the matrix inversion in practical implementation, or it might be too expensive to perform at all. In such circumstances we recommend using an \emph{ online} version of the algorithm, where two samplers are run in parallel, one sampling the gradient mini-batch and one sampling the information for the inverse Hessian estimator $\bR^k$. 
			
			Should the two sampling procedures be uncorrelated, and the gradient estimate unbiased such as $\sum_{j \in S_2}\nabla f_j(x_k) $, where elements in $S_2 \subset [n]$ with all elements having an equal probability of being sampled, we are able to prove again in expectation (on these sampling processes) convergence. However, the statement of the variance remains intractable. We perform a numerical experiments to explore the impact of batch implementation in the case of quadratic function. 
			
			We perform several runs of ISSA algorithm, with the full gradient, and online batch gradient with a separate independent sampling $\hat{S}_2$. Numerical experiments reveal that when the sampling mechanism is shared with the $\bR^k$, the algorithm is unstable. Also, we report that the algorithm tends to be unstable when $|S^k|>|\hat{S}_2|$. Therefore, we suggest an adaptive scheme where $|\hat{S}_2|$ grows to improve the gradient estimate. With increasing $|\hat{S}_2|$ size the optimization stabilizes faster, and the instability occurs in later stages of optimization the optimization. For details please explore Figure \ref{fig:2}. 
			
			\begin{figure}
				\centering
				
					\begin{subfigure}[t]{0.49\textwidth}
						\includegraphics[width=\textwidth]{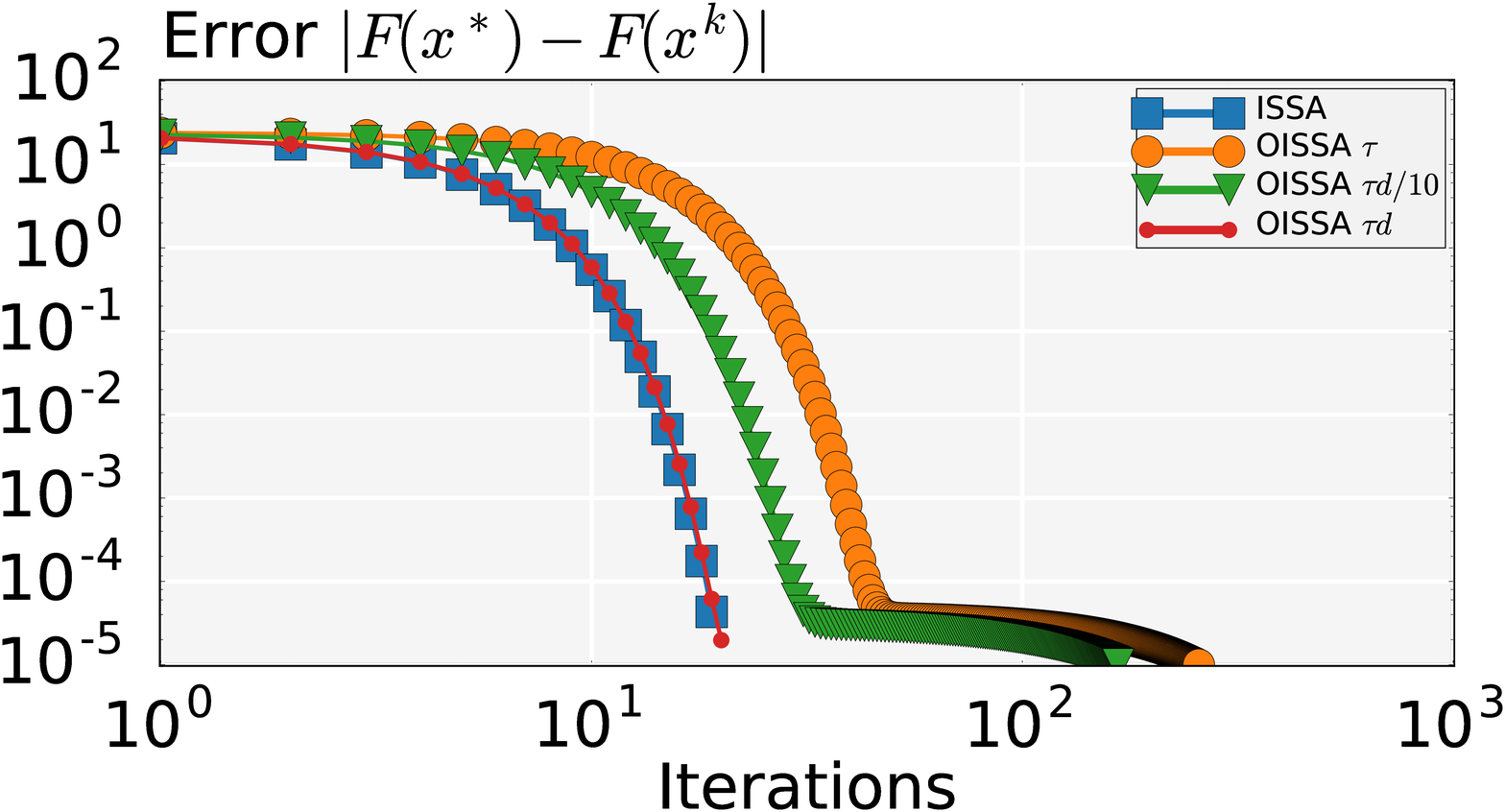}
						\caption{A linear regression model with an artificial dataset such that $n = 10^6$ and $d = 10^2$ with batch gradient estimator. Different batch-size for gradient presented in the legend}
						\label{fig:gull2}
					\end{subfigure}
					~
					\begin{subfigure}[t]{0.49\textwidth}
						\includegraphics[width=\textwidth]{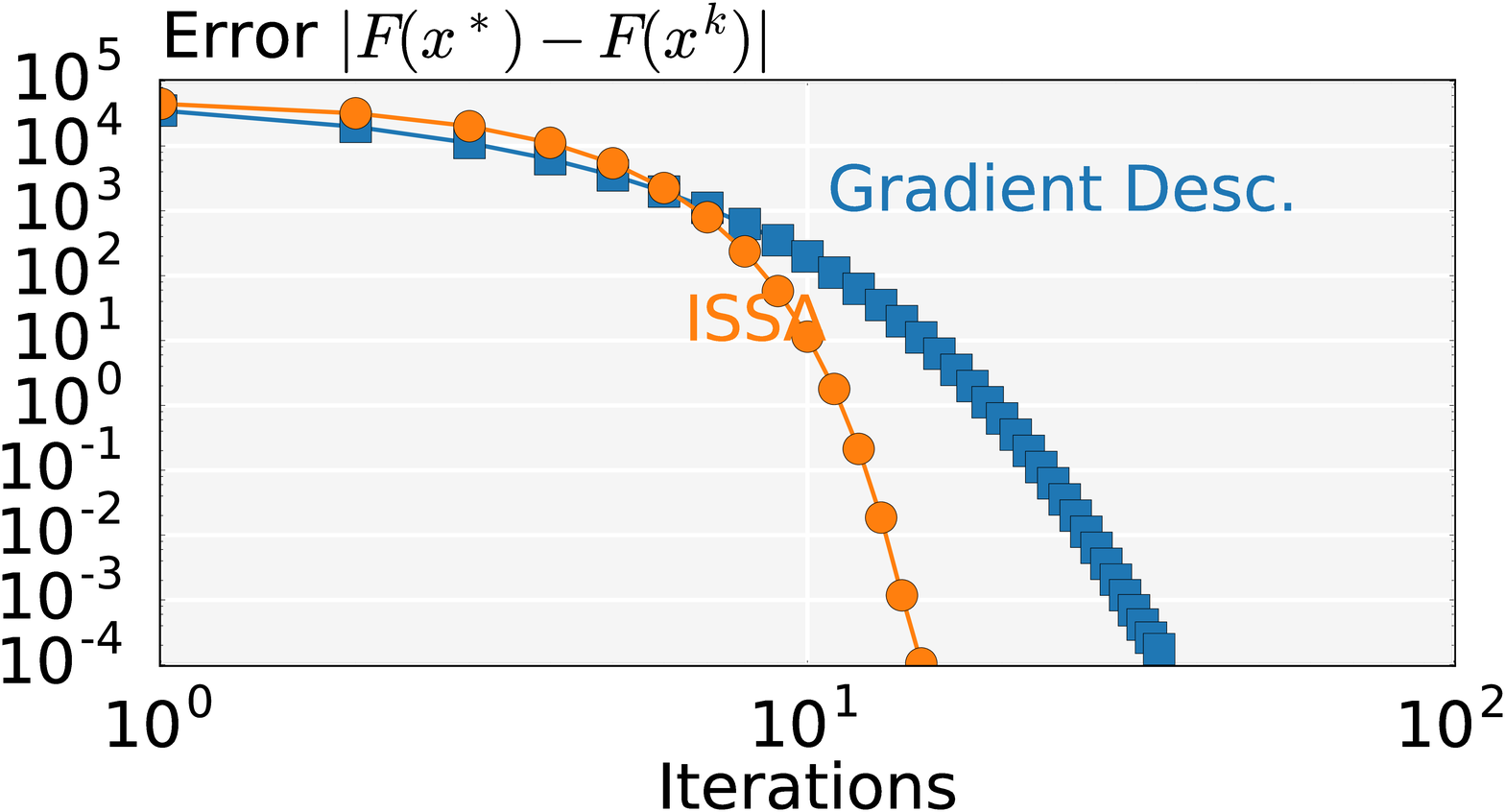}
						\caption{A linear regression model with an artificial dataset such that $n = 10^6$ and $d = 10^2$. The advantage of scalling become appartent very soon.}
						\label{fig:gull3}
					\end{subfigure}
					
				\label{fig:2}
			\end{figure}

		\subsection{Adaptive Settings and Diminishing History Property}
			In online convex optimization, where data points arrive in streams, one can consider a scenario where the probabilistic model with which the data points are generated changes over time. In such circumstances, the optimal point moves $x^{*,t}$ and is differs the optimization. The desired effect of algorithm is to tracks the trajectory of the optimum in time.
			
			Standard theory of stochastic gradient descent requires that the step-sizes of the algorithm be decreasing such that $\sum_{j}^{\infty} \gamma_t = \infty$, but at the same time $\sum_{j}^{\infty} \gamma_t^2 \leq \infty$. Adaptive algorithms such as AdaGrad \cite{Duchi2011} incorporate information about the objective on the flight, however irrespective of the counter, all information is equally weighted. When the objective is changing, the old information looses its importance, and this is not reflected in either of the algorithms.  
			
			In contrast to this, ISSA posses a property that it adapts the geometry of the step with the information from Hessian samples, and weights the older updates less than the newer ones. This leads us to conjecture that ISSA may be suitable for online problems where the probabilistic model is changing over the course of optimization, and an adaptive behavior is desired. 
		
	\section{Numerical Results}
	The experiment was performed on Lenovo Laptop with processor Intel i7 2.6 GHz. We used the numerical linear algebra library Eigen as the main backbone for our implementation. We show one experiment fitting least squares estimator in Figure \ref{fig:1} with artificial dataset where the design matrix elements are sampled from truncated standard normal distribution, and logistic regression on standard \emph{mushroom} dataset \ref{alg:Method1}.
	
	\begin{figure}
		\centering
		
		\begin{subfigure}[b]{0.49\textwidth}
			\includegraphics[width=\textwidth]{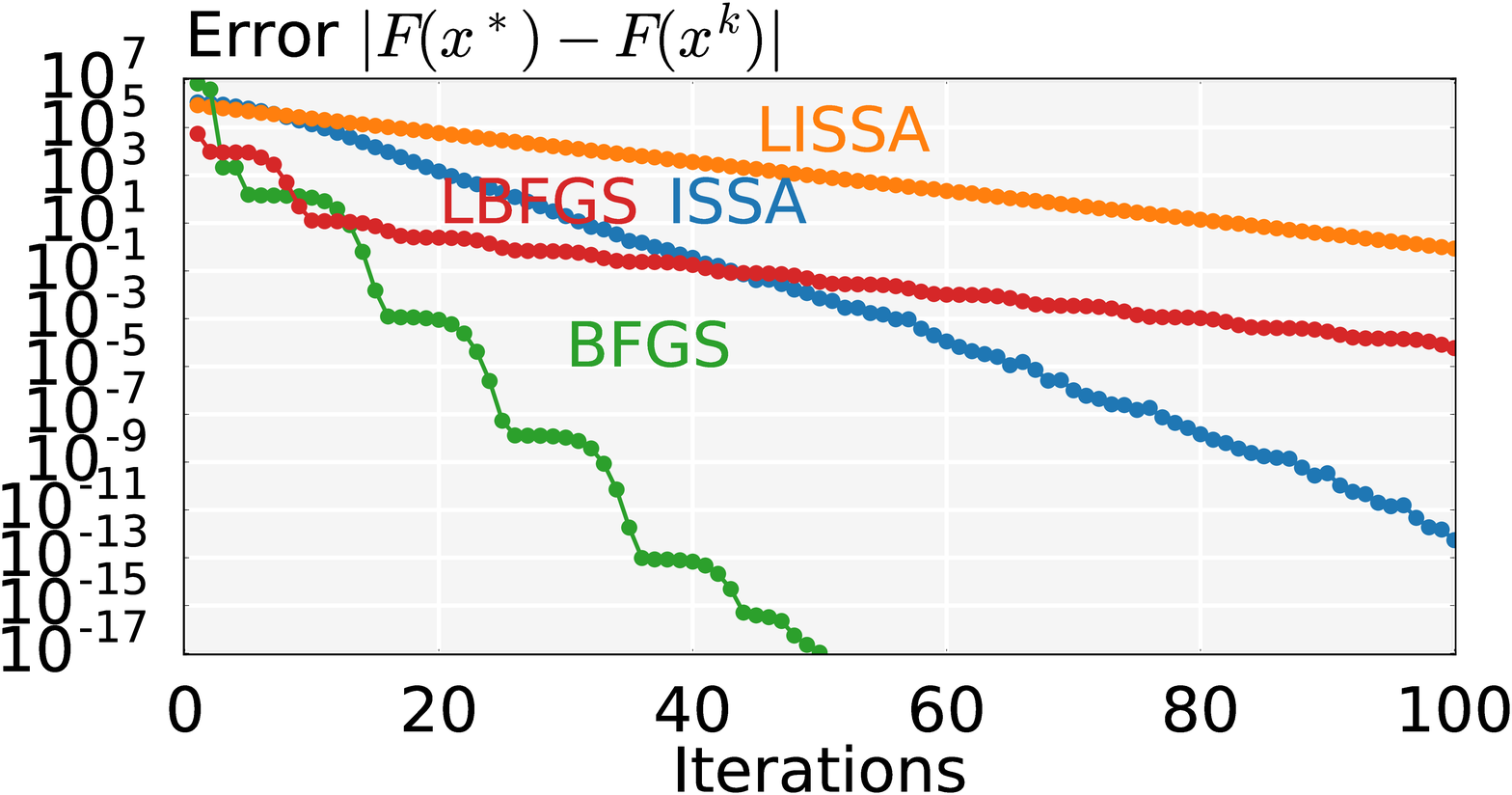}
			\caption{A linear regression model with an artificial dataset such that $n = 10^6$ and $m = 10^2$.}
			\label{fig:gull}
		\end{subfigure}
		~ 
		\begin{subfigure}[b]{0.49\textwidth}
			\includegraphics[width=\textwidth]{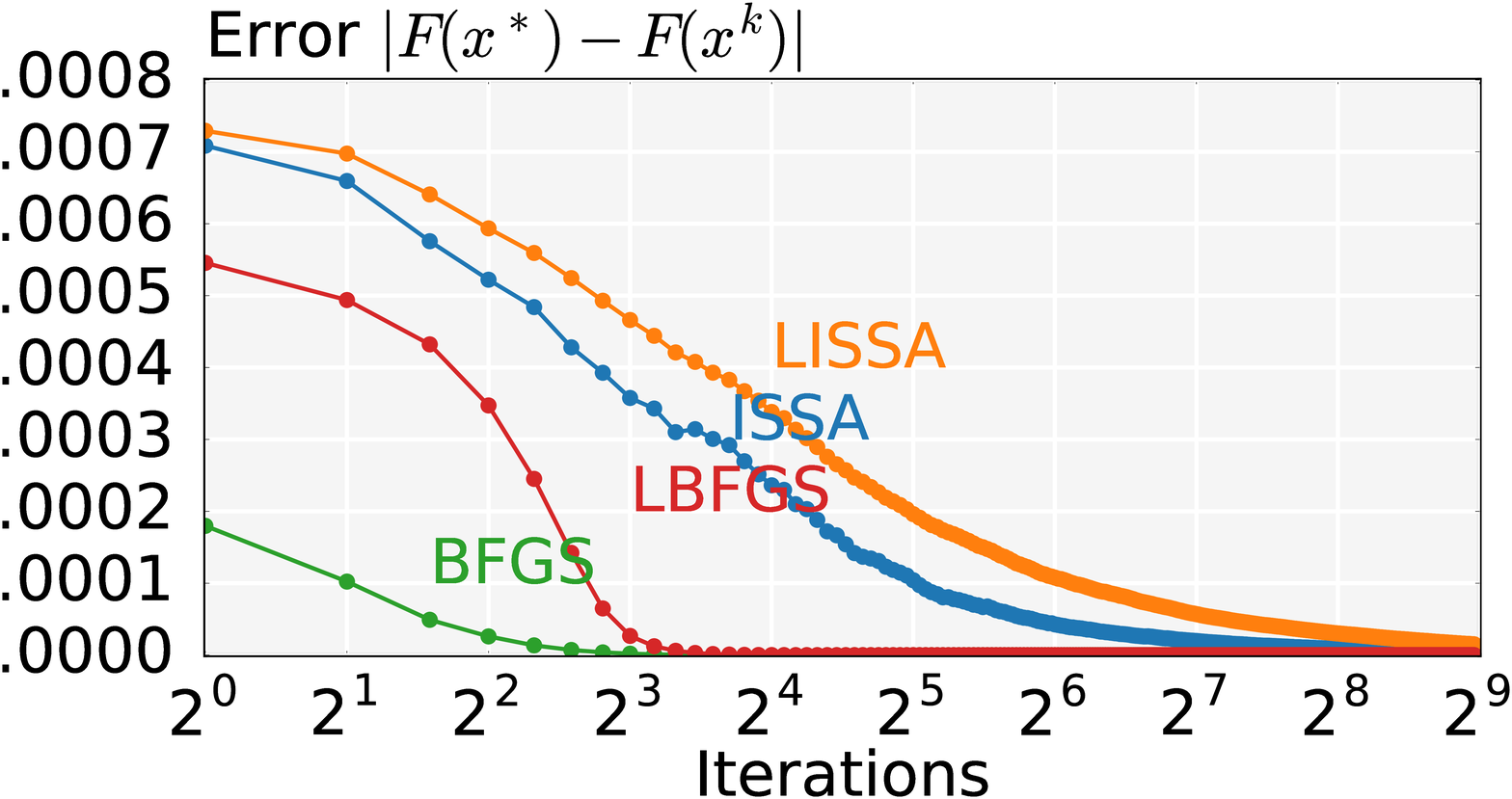}
			\caption{A logistic regression on \emph{mushroom} dataset with $n = 8124$ and $m = 22$. }
			\label{fig:tiger}
		\end{subfigure}

		\caption{A numerical experiment comparing ISSA, LISSA, BFGS and L-BFGS algorithms. In a) the $\tau = 5$ for ISSA and L-BFGS used last 5 gradient information as well. LISSA was run with $S_2 = 20, S_1 = 1$. In the second experiment on the real dataset all parameters were equal to $3$. ISSA was run without the decreasing step size. }
		\label{fig:1}
	\end{figure}

	\section{Conclusion}
	We presented an iterative stochastic second-order optimization algorithm that iteratively builds an approximation to Hessian inverse via Neumann Series. We have provided two analyses of the convergence properties of the algorithm. One of those shows linear convergence under milder conditions than previous results and one analysis which might shed light on  the very fast initial convergence of the method. The numerical experiment shows that our method matches L-BFGS performance.

	\subsubsection*{Acknowledgments}
	The author would like to thank Prof Nicolai Meinshausen of ETH Z\"urich, and Prof Martin Jaggi of EPF Lausanne for their kind advice and help with this project.
	\bibliographystyle{plain}
	\bibliography{bib.bib}

	\section*{Appendix}
	
	\begin{lemma}
			Let $\mathbf{R}_k$ be defined as in Definition \ref{def:2}. Then the following bound on the first moment holds,
			\begin{eqnarray}
				\lambda_{max} ( -\mathbb{E}(\mathbf{R}_k)) &  \leq &  - \frac{(1-\alpha)^2(1-(1-\alpha)^{2k\tau - 2})}{2\alpha - \alpha^2}
				\label{eq:22} \\
				 & \leq & - \frac{(1-\alpha)^2}{\alpha(2-\alpha)}+\mathcal{O}((1-\alpha)^{k\tau}).
			\end{eqnarray}
	\end{lemma}
	\begin{proof}
		\begin{eqnarray}
		\lambda_{max} ( -\mathbb{E}(\mathbf{R}_k)) & =  & \max_{\norm{x}=1} -\braket{x,\mathbb{E}(\mathbf{R}_k)x} \\
		& \stackrel{\eqref{eq:Rk}} = &  \max_{\norm{x}=1} \braket{x,-\sum_{j=1}^{k\tau}(\bI - \nabla^2 f )^j x} \\
		&  = &  \max_{\norm{x}=1} \left( \braket{x,\sum_{j \text{ odd}}^{k\tau}( \nabla^2 f - \bI )^j - \sum_{j \text{ even}}^{k\tau}(\bI - \nabla^2 f )^j x} \right) \\
		& \leq & \max_{\norm{x}=1} \braket{x,\sum_{j \text{ odd}}^{k\tau}( \nabla^2 f - \bI )^jx} - \min_{\norm{x}=1} \braket{x,\sum_{j \text{ even}}^{k\tau}(\bI - \nabla^2 f )^jx} \\ 
		& \stackrel{\eqref{eq:strgcnvx}} \leq & -\min_{\norm{x}=1}  (\braket{x,\sum_{j \text{ even}}^{k\tau}(\bI - \nabla^2 f )^j x} ) \\
		& \stackrel{\eqref{eq:strgcnvx}} \leq & - \sum_{j \text{ even}}^{k\tau}(1 - \alpha )^j \\
		&  = & - \frac{(1-\alpha)^2 (1 - (1-\alpha)^{2k\tau - 2}) }{2\alpha - \alpha^2} \\
		& \leq & - \frac{(1-\alpha)^2}{\alpha(2-\alpha)} + \frac{(1-\alpha)^{2k\tau}}{\alpha(2-\alpha)}
		\end{eqnarray}
	\end{proof}

	\begin{lemma}
		Let $\mathbf{R}_k$ be defined as in Definition \ref{def:2}. Then the following bound on the second moment holds, 
		\begin{equation}
			\lambda_{max}(\mathbb{E}[(\mathbf{R}_k)^\top \mathbf{R}_k] ) \leq \frac{2-\alpha}{\alpha^2} + (1-\alpha)^{k\tau}. \label{eq:231}
		\end{equation}
	\end{lemma}
	\begin{proof}
		We use shorthand for $v_{k+1} = \lambda_{max}(\mathbb{E}[(\mathbf{R}_{k+1})^\top \mathbf{R}_{k+1}])$, $c = 1+\frac{2}{\alpha}$ and $\sigma_{k+1} = v_{k+1} + \frac{2}{\alpha}$. Also, we know that $v_1 \leq 1$. Then using the definition of update in Algorithm \ref{alg:Method1}, we get,
		\begin{eqnarray}
			(\mathbf{R}_{k+1})^\top \mathbf{R}_{k+1} & = &  (\bI + \bR_k^\top(\bI - \bX_j)^\top)(\bI +  (\bI - \bX_j)\bR_k) \\
				& = & \bI + \bR_k^\top(\bI - \bX_j)^\top + (\bI - \bX_j)\bR_k + \bR_k^\top(\bI-\bX_j)^2\bR_k \\
				& \stackrel{Ass \ref{ass:scaled}} \leq & \bI + \bR_k^\top(\bI - \bX_j)^\top + (\bI - \bX_j)\bR_k + \bR_k^\top(\bI-\bX_j)\bR_k \label{eq:222} \\
			v_{k+1} & \stackrel{(v_k),\eqref{eq:strgcnvx}}\leq &  1 + (1-\alpha	)\left(\frac{2}{\alpha} + v_k\right) \\
			v_{k+1} + \frac{2}{\alpha} & \leq & 1 + \frac{2}{\alpha} + (1-\alpha) \left(\frac{2}{\alpha} + v_k\right)\\
			\sigma_{k+1} & \leq & c + (1-\alpha)\sigma_{k} \\
			 & \leq & c\sum_{i=0}^{k\tau-1}(1-\alpha)^i + (1-\alpha)^{k\tau} \\
			 & \leq & \frac{c}{\alpha} + (1-\alpha)^{k\tau}
		\end{eqnarray}
		We remark that given $k$ large it is reasonable to assume stationary and we can check that the calculated bound, assuming stationarity from \eqref{eq:222} matches the non-stationary result. 
	\end{proof}

	\begin{proof}[Proof of Theorem 1]
			We adopt the notation shorthand that $\nabla f(x^k) = g^k$,
			\begin{equation}
			c := \frac{\beta (2-\alpha)^2 + \beta(2-\alpha)(1-\alpha)^{k\tau}}{ (1-\alpha)^2 - (1-\alpha)^{2k\tau}},
			\end{equation}
			and for $k \rightarrow \infty$. Also, we define $c_{\infty}  := \frac{\beta (2-\alpha)^2}{ (1-\alpha)^2}$.			\begin{eqnarray*}
					f(x^{k+1})-f(x^k)  & \stackrel{\eqref{eq:strgcnvx}}\leq &  -\frac{1}{c} \braket{g^k, \mathbf{R}^k g^k} +  \frac{\beta}{2c^2}\braket{ \mathbf{R}^kg^k, \mathbf{R}^k g^k} 	\\
					\mathbb{E}[f(x^{k+1})-f(x^k)]  &  \leq &  -\frac{1}{c} \braket{g^k, \mathbb{E}[\mathbf{R}^k] g^k} +  \frac{\beta}{2c^2}\braket{ g^k, \mathbb{E}[(\mathbf{R}^k)^\top\mathbf{R}^k] g^k} 	\\
					&  \stackrel{\eqref{eq:231},\eqref{eq:22}} \leq & -\frac{  (1-\alpha)^2(1-(1-\alpha)^{2k\tau - 2})  }{c(2\alpha - \alpha^2)} \braket{g^k, g^k} + \frac{\beta}{2c^2}\left(  \frac{(2-\alpha)}{\alpha^2} + (1-\alpha)^{k\tau}\right)\braket{g^k,g^k} 	\\
					& \stackrel{(c)}\leq & -\frac{ (1-\alpha)^4 }{2\beta ( 2-\alpha   )^2 ( (2-\alpha) + \alpha^2(1-\alpha)^{k\tau}}  \braket{g^k,g^k}	\\
					& \stackrel{\eqref{eq:strgcnvx}} \leq &  -\frac{ (1-\alpha)^4 \alpha}{\beta ( 2-\alpha   )^2 ( (2-\alpha) + \alpha^2(1-\alpha)^{k\tau}} (f(x^k) - f(x^*)) \\
					& \stackrel{\eqref{eq:mu}} \leq & -\mu (f(x^k) - f(x^*))
			\end{eqnarray*} 
	\end{proof}
	
	\begin{proof}[Proof of Lemma 1 ]
		Using von Neumann expansion  
		\begin{eqnarray}
			\norm{ \mathbb{E}[\mathbf{R}^k] - (\nabla^2 f(x))^{-1} }_2^2 & \stackrel{\eqref{eq:234}}= & \norm{\sum_{j = k\tau}^{\infty}  (\mathbf{I} - \nabla^2 f(x))^{-1} } \\
			& \leq & (1-\alpha)^{\tau k }\sum_{k=0}^{\infty} \norm{(\mathbf{I} - \nabla^2 f(x))^{-1} }_2 \\
			& \leq &  \frac{(1-\alpha)^{k\tau}}{\alpha}
		\end{eqnarray}
	\end{proof}\hfil
	\begin{proof}[Proof of Theorem 2]
		\begin{equation}
			\Delta x^k := -\mathbf{R}^k \nabla f(x^k)
		\end{equation}
		\begin{eqnarray*}
			\norm{\nabla f(x^{k+1})} & = & \norm{\nabla f(x^{k+1}) + \nabla f(x^{k}) - (\mathbf{R}^k)^{-1}\Delta x^k} \\
			& \leq & \norm{\nabla f(x^{k+1}) + \nabla f(x^{k}) +  \left( \nabla^2 f(x^{k}) - \nabla^2 f(x^{k})- (\mathbf{R}^k)^{-1}  \right)\Delta x^k  }\\	
			&  \leq & \norm{\nabla f(x^{k+1}) + \nabla f(x^{k}) -  \nabla^2 f(x^{k}) \Delta x^k} + \norm{ \left(\nabla^2 f(x^{k}) - (\mathbf{R}^k)^{-1}  \right)\Delta x^k} \\
			& \leq & 2\beta \norm{\Delta x ^k}^2 + \norm{(\mathbf{R}^k)^{-1}} \norm{\nabla^2 f(x^k)}\norm{\nabla^2 f(x^{k})^{-1} - \mathbf{R}^k  } \norm{\Delta x^k} \\
			& \leq & 2\beta \norm{\Delta x ^k}^2 + \beta^2 \norm{\nabla^2 f(x^{k})^{-1} - \mathbf{R}^k  } \norm{\Delta x^k} \\
			& \leq & 2\beta \norm{\Delta x ^k}^2 + \beta^2 \left( \norm{\nabla^2 f(x^{k})^{-1} - \mathbb{E}[\mathbf{R}^k]  } + \norm{\mathbb{E}[\mathbf{R}^k] - \mathbf{R}^k  } \right) \norm{\Delta x^k}\\
			& \leq & 2\frac{\beta}{\alpha^2} \norm{\nabla f(x^k)}^2 + \frac{\beta^2}{\alpha} \left( \norm{\nabla^2 f(x^{k})^{-1} - \mathbb{E}[\mathbf{R}^k]  } + \norm{\mathbb{E}[\mathbf{R}^k] - \mathbf{R}^k  } \right) \norm{\nabla f(x^k)}\\
			\mathbb{E}[\norm{\nabla f(x^{k+1})}] & \stackrel{\eqref{eq:approx}} \leq & 2\frac{\beta}{\alpha^2} \norm{\nabla f(x^k)}^2 + \frac{\beta^2}{\alpha} \left( \frac{(1-\alpha)^{k\tau}}{\alpha} + \frac{\beta-\alpha}{2}   \right) \norm{\nabla f(x^k)} \\
			& = & 2\frac{\beta}{\alpha^2}\norm{\nabla f(x^k)} \left( \norm{\nabla f(x^k)} +  \frac{\beta (1-\alpha)^{k\tau}}{2} + \frac{\alpha \beta (\beta-\alpha)}{4}   \right) \\
			& \stackrel{\eqref{eq:cond}} \leq & 4\frac{\beta}{\alpha^2}\norm{\nabla f(x^k)}^2
		\end{eqnarray*}
		Lastly, \eqref{eq:strgcnvx2} implies,
		\begin{eqnarray}
			\mathbb{E}[f(x^k) - f(x^l)] \leq \frac{\alpha}{8 \beta} \left(\frac{1}{2}\right)^{2^{(k-l)}},
		\end{eqnarray}
		where $l$ is the index such that the \eqref{eq:cond} still holds.
	\end{proof}

\end{document}